\documentclass[11pt]{article}

\usepackage{amssymb,amsmath,amsthm}
\usepackage{graphicx}
\usepackage{multicol}
\usepackage{color}

\def\blu{\textcolor{blue}}

\def\R{\mathbb{R}}
\def\Z{\mathbb{Z}}
\def\i{\mathrm i}
\def\d{\mathrm d}
\def\e{\mathrm e}
\def\E{\mathbb E}
\def\P{\mathbb P}
\def\text{\mbox}

\def\1{{\bf 1}}

\newcommand{\noi}{\noindent}
\newcommand {\nn}{\nonumber}

\newtheorem{theorem}{Theorem}
\newtheorem{corollary}[theorem]{Corollary}

\newtheorem{proposition}[theorem]{Proposition}

\theoremstyle{definition}

\newtheorem{remark}{Remark}

\usepackage[authoryear]{natbib}
\begin{document}

\title{Estimating long memory in panel
 random-coefficient AR(1) data}
\author{Remigijus Leipus$^{1}$, \ Anne Philippe$^{2}$, \ Vytaut\.e Pilipauskait\.e$^{3}$, \  Donatas Surgailis$^1$  }
\date{\today \\ \small
\vskip.2cm
$^1$Vilnius University, Faculty of Mathematics and Informatics, Institute of Applied Mathematics, Lithuania\\
$^2$Universit\'{e} de Nantes, Laboratoire de Math\'{e}matiques Jean
Leray, France \\
$^3$Aarhus University, Department of Mathematics, Denmark.  \\
}

\maketitle

\begin{abstract}
	
We construct an asymptotically normal estimator $\widetilde \beta_N$ for the tail index $\beta$ of a distribution on $(0,1)$ regularly varying at $x=1$,
when its $N$ independent realizations are not directly observable. 
The estimator $\widetilde \beta_N$  is a version of the tail index estimator of Goldie and Smith (1987) based on suitably truncated observations contaminated with arbitrarily dependent `noise' which vanishes as $N$ increases.
We apply $\widetilde{\beta}_N$ to panel data comprising
$N$  random-coefficient AR(1) series, each of length $T$, for estimation of the tail index of the random coefficient at the unit root, in which case the unobservable random coefficients are replaced
by sample lag 1 autocorrelations of individual time series.
Using asymptotic normality of $\widetilde \beta_N$, we construct a statistical procedure to test if the panel random-coefficient AR(1) data exhibit long memory.
A simulation study illustrates finite-sample performance of the introduced inference procedures.

\end{abstract}

\medskip


\section{Introduction}

Dynamic panels (or longitudinal data) comprising observations taken at regular time intervals for the same individuals such as households, firms, etc.\ in a large heterogeneous population, are often described by time series models with random parameters (for reviews on dynamic panel data analysis, see \cite{arel2003}, \cite{balt2015}).
One of the simplest models for individual evolution is the random-coefficient AR(1) (RCAR(1)) process
\begin{equation} \label{RCAR1}
X_i(t) = a_i X_i (t-1) + \zeta_i (t), \quad t \in \Z, \quad i=1, 2, \dots,
\end{equation}
where the innovations $\zeta_i(t)$, $t \in \Z$, are independent identically distributed (i.i.d.) random variables (r.v.s) with $\E \zeta_i (t) = 0$, $\E \zeta^2_i (t) < \infty$ and the autoregressive coefficient $a_i \in (0,1)$ is a r.v.,
independent of $\{\zeta_i (t), \, t \in \Z\}$. It is assumed that the random coefficients $a_i$, $i=1,2, \dots$, are i.i.d., while the innovation sequences $\{\zeta_i(t),t\in \Z\}$ can be either independent or dependent across $i$, by inclusion of a common `shock'
to each unit; see Assumptions  (A1)--(A4) below.
If the distribution of $a_i$ is sufficiently `dense' near unity,
then statistical properties of the individual evolution in \eqref{RCAR1}  and the corresponding panel
can differ greatly from those in the case of fixed $a \in (0,1)$. To be more specific, assume that the AR coefficient
$a_i$ has a density function $g (x)$, $x \in (0,1)$, satisfying
\begin{equation}\label{betacond}
g(x) \sim g_1 (1-x)^{\beta - 1}, \quad x \to 1-,
\end{equation}
for some $\beta >1$ and $g_1 > 0$. Then a stationary solution of RCAR(1) equation \eqref{RCAR1} has the following autocovariance function
\begin{equation}\label{covX}
\E X_i (0) X_i(t) =  \E \zeta^2_i (0) \E \frac{a^{|t|}_i}{1-a^2_i} \sim \frac{g_1}{2}
 \Gamma(\beta-1) \E \zeta^2_i (0)  t^{-(\beta-1)}, \quad
t \to \infty,
\end{equation}
and exhibits long memory in the sense that
$\sum_{t\in \Z} |\operatorname{Cov} (X_i(0), X_i(t))| = \infty$  for $\beta \in (1,2]$.
The same long memory property applies to the contemporaneous aggregate
\begin{equation} \label{barXN}
\bar X_N(t) := N^{-1/2} \sum_{i=1}^N X_i(t), \quad t \in \Z,
\end{equation}
of  $N$ independent individual evolutions in 
\eqref{RCAR1} and its Gaussian
limit arising as $N \to \infty$.
For the beta distributed squared AR coefficient $a_i^2$, these facts were first uncovered by \cite{gran1980} and later
extended to more general distributions and/or RCAR equations in \cite{gonc1988}, \cite{zaff2004}, \cite{celo2007}, \cite{OV2004}, \cite{ps2010}, \cite{PPS2014} and other works, see \cite{lei2014} for review. Assumption \eqref{betacond} and the parameter $\beta $
play a crucial role for statistical  (dependence) properties of
the panel $\{ X_i (t), \, t = 1, \dots, T, \, i = 1, \dots, N \}$ as $N $ and $T $ increase, possibly at different  rates.
Particularly, \cite{pil2014} proved that for $\beta \in (1, 2)$
the distribution of the normalized sample mean
$\sum_{i=1}^N \sum_{t=1}^T X_i(t)$ is asymptotically normal if $N / T^\beta \to \infty$ and $\beta$-stable if $N / T^\beta \to 0$ (in the
`intermediate' case  $N / T^\beta \to c \in (0,\infty)$ this limit distribution is more complicated and given by
an integral with respect to a certain Poisson random measure). In the case of common innovations ($\{\zeta_i(t), \ t \in \Z\} \equiv \{\zeta(t), \ t \in \Z\} $) the limit stationary aggregated process exists
under a different normalization ($N^{-1} $ instead of $N^{-1/2}$ in \eqref{barXN})
and is written as a moving-average in the above innovations
with deterministic coefficients $\E a_1^j$, $j \ge 0$, which decay as $\Gamma (\beta) j^{-\beta} $ with $j \to \infty $
and exhibit long memory for $ \beta \in (1/2, 1)$;
see \cite{zaff2004}, \cite{ps2009}.
The trichotomy of the limit distribution of the sample mean for a panel comprising RCAR(1) series driven by common innovations is discussed in \cite{pil2015}.

In the above context, a natural
statistical problem concerns inference about
the distribution of the random AR coefficient $a_i$, e.g., its cumulative distribution function (c.d.f.) $G$  or the parameter $\beta $ in \eqref{betacond}.
\cite{lei2006}, \cite{celo2010} estimated the density $g$ using sample autocovariances of the limit aggregated process.
For estimating parameters of $G$, \cite{bib:ROB78} used the method of moments. He proved asymptotic normality of the estimators for moments of $G$ based on the panel RCAR(1) data as $N \to \infty$
for fixed $T$, under the condition  $\E (1-a^2_i)^{-2} < \infty$ which does not allow for long memory in $\{ X_i(t), \, t \in \Z \}$.
For parameters of the beta distribution,
\cite{ber2010} discussed maximum likelihood estimation based
on (truncated) sample lag 1 autocorrelations computed from $\{X_i(1), \dots, X_i(T) \}$, $i=1, \dots, N$,
and proved consistency and asymptotic normality of the introduced estimator as $N, T \to \infty$. In nonparametric context,
\cite{lei2016} 
studied the  empirical c.d.f.\ of $a_i$ based on 
sample lag 1 autocorrelations similarly to  \cite{ber2010},
and derived  its asymptotic properties as $N, T \to \infty$, including those of
a kernel density estimator.  
Moreover, \cite{lei2016} proposed another estimator of moments of $G$ 
and proved its asymptotic normality as $N, T \to \infty$. Except for parametric situations,
the afore mentioned results do not allow for inferences about the tail parameter $\beta $ in \eqref{betacond} and
testing for the presence or absence of long memory in panel RCAR(1) data.

The present paper discusses in semiparametric context, the estimation of $\beta$ in \eqref{betacond} from RCAR(1) panel  $\{ X_i (t), \, t = 1, \dots, T, \, i = 1, \dots, N \}$ with finite variance $\E X_i^2 (t) < \infty$. We use the fact that \eqref{betacond} implies
$\P ( 1/(1-a_i) > y) \sim (g_1/\beta) y^{-\beta}$, $y \to \infty$, i.e.\ r.v.\ $1/(1-a_i)$ follow a heavy-tailed distribution with index $\beta >1$. Thus, if $a_i$, $i=1,\dots, N$, were observed, $\beta $ could be estimated by a number of tail index estimators, including the Goldie and Smith
estimator in \eqref{best} below.
Given panel data, the unobservable $a_i$ can be estimated by sample lag 1 autocorrelation $\widehat a_{i}$ computed from $\{X_i(1), \dots, X_i(T)\}$
for each $i=1, \dots, N$. This leads to the general estimation problem of $\beta $ for 
`noisy' observations
\begin{equation}\label{hata}
\widehat a_i = a_i + \widehat \rho_i, \qquad i=1, \dots, N,
\end{equation}
where the `noise', or measurement error $\widehat \rho_i = \widehat a_{i} - a_i$ is of unspecified nature and
vanishes with $N \to \infty $.

Related statistical problems where observations contain measurement error
were discussed in several papers. \cite{resn1997}, \cite{lin2004} considered Hill estimation of the tail parameter from residuals of ARMA series.
Kim and Kokoszka (\citeyear{kim2019a}, \citeyear{kim2019b}) discussed asymptotic properties and finite sample performance of Hill's estimator for observations contaminated with i.i.d.\ `noise'.  The last paper contains further references on inference problems with measurement error.

A major distinction between the above mentioned works
and our study is that we estimate the tail behavior of $G$ at a finite point $x=1 $ and
therefore the measurement error should vanish with $N$
which is not required in Kim and Kokoszka (\citeyear{kim2019a}, \citeyear{kim2019b})
dealing with estimation of the tail index at infinity.
On the other hand, except for the `smallness condition' in  \eqref{condrhoN}--\eqref{condN}, no other (dependence or independence)
conditions on the `noise' in \eqref{hata} are assumed, in contrast to Kim and Kokoszka (\citeyear{kim2019a}, \citeyear{kim2019b}),
where the measurement errors are i.i.d.\ and independent of the `true' observations.
The proposed estimator
$\widetilde \beta_N$ in \eqref{best_tilde} is a `noisy' version of the Goldie and Smith estimator, applied to
 observations in \eqref{hata} truncated at a level close to 1.
The main result of our paper is Theorem~\ref{thm:2N} giving sufficient conditions
for asymptotic normality of the  constructed estimator
$\widetilde \beta_N$.  These conditions involve $\beta$ and other asymptotic parameters of $G$ at $x=1$
and the above-mentioned `smallness' condition restricting the choice of the
threshold parameter $\delta = \delta_N \to 0$ in $\widetilde \beta_N$. Theorem~\ref{thm:2N} is applied
to the RCAR(1) panel data, resulting in an asymptotically normal estimator of $\beta $,
where the `smallness condition' on the `noise' is verified provided  $T = T_N$ grows
fast enough with $N$ (Corollary \ref{thm:2}).
Based on the above asymptotic result, we construct a statistical procedure to test the presence of long memory in the panel, more precisely, the null hypothesis $H_0 : \beta \ge 2$ vs.\ the long memory alternative $H_1 : \beta \in (1,2)$.

The paper is organized as follows.
Section~\ref{sec2} contains the definition of the estimator $\widetilde \beta_N$ and
the main  Theorem~\ref{thm:2N} about its asymptotic normality for `noisy' observations.
Section \ref{sec3} provides the
assumptions on  the RCAR(1) panel model, together with application of  Theorem~\ref{thm:2N}
based on the panel data and some consequences.
In Section~\ref{sec4} a simulation study
illustrates finite-sample properties of the introduced estimator
and the testing procedure. Proofs can be found in Section~\ref{sec5}.

In what follows, $C$ stands for a positive constant whose precise value is unimportant and which may change from line to line. We write  $\to_p, \, \to_d$ for the convergence in probability and distribution respectively, whereas $\to_{D[0,1]}$ denotes the weak convergence in the space $D[0,1]$ with the uniform metric. Notation ${\cal N}(\mu, \sigma^2)$ is used for the normal distribution with mean $\mu$ and variance $\sigma^2$.

\section{Estimation of the  tail parameter from `noisy' observations} \label{sec2}

In this section we introduce 
an estimator of the tail parameter $\beta$ in \eqref{betacond} based on `noisy'
observations in \eqref{hata},
where $a_i \in (0,1) $ are i.i.d.\ satisfying \eqref{betacond}, and $\widehat \rho_i = \widehat \rho_{i, N}$ are measurement errors (i.e., arbitrary random variables)
which vanish with $N \to \infty$ at a certain rate, uniformly in $i=1, \dots, N$.

To derive asymptotic results about this estimator, condition \eqref{betacond}  is strengthened as follows.

\bigskip

\noi {\bf  (G)}\ \ $a_i \in (0,1)$, $i=1,2,\ldots$, are independent r.v.s with common
c.d.f.\
$G(x) := \P (a_i \le x)$, $x \in [0,1]$.
There exists $\epsilon \in (0, 1)$ such that $G$ is continuously differentiable on $(1 - \epsilon, 1)$ with derivative satisfying
\begin{equation}\label{cond:G}
g(x) = \kappa \beta (1-x)^{\beta-1} ( 1 + O((1-x)^\nu) ), \quad x \to 1-,
\end{equation}
for some $\beta > 1$, $\nu > 0$ and $\kappa > 0$.

\bigskip

Assumption (G) implies that the tail of the c.d.f.\ of $Y_i := 1/(1-a_i)$ satisfies
\begin{equation}\label{tail}
\P (Y_i > y) = \kappa y^{-\beta} (1 + O(y^{-\nu})), \quad y \to \infty.
\end{equation}

For independent observations $Y_1, \dots, Y_N$ with common c.d.f.\ satisfying \eqref{tail},
\cite{gol1987} introduced the following estimator of the tail index $\beta$:
\begin{equation}\label{est}
\beta_N := \frac{\sum_{i=1}^N \1 (Y_i \ge v)}{\sum_{i=1}^N \1 (Y_i \ge v) \ln (Y_i/v)}, 
\end{equation}
and proved asymptotic normality of this estimator provided the threshold level $v = v_N$ tends to infinity at an appropriate rate as $N \to \infty$.

For independent realizations $a_1, \ldots, a_N$ under assumption~(G), we rewrite the tail index estimator in \eqref{est} as
\begin{equation}\label{best}
\beta_N = \frac{ \sum_{i=1}^N \1 (a_i >1- \delta) }{  \sum_{i=1}^N \1 (a_i >1- \delta) \ln (\delta/(1- a_i))},
\end{equation}
where $\delta := 1/v$ is a threshold close to 0.

\begin{theorem}\label{thm:1}
Assume (G). If $\delta = \delta_N \to 0$ and $N \delta^\beta \to \infty$ and $N \delta^{\beta + 2 \nu} \to 0$ as $N \to \infty$, then
\begin{equation*}
\sqrt{N \delta^{\beta}} (\beta_N - \beta ) \to_d {\cal N}(0,\beta^2/\kappa).
\end{equation*}
\end{theorem}

Theorem \ref{thm:1} is due
to Theorem 4.3.2 in \cite{gol1987}. The proof in \cite{gol1987}
uses Lyapunov's CLT conditionally on the number of exceedances over a threshold. Further sufficient conditions for
asymptotic normality of $\beta_N$ were obtained in \cite{nov1990}.
In Section~\ref{sec5} we give an alternative proof of Theorem~\ref{thm:1} based on the tail empirical process.
Our proof has the advantage that it can be more easily adapted to
prove asymptotic normality of the `noisy' modification  of  \eqref{best} defined as
\begin{equation}\label{best_tilde}
\widetilde \beta_N := \frac{ \sum_{i=1}^N \1 (\widetilde a_i >1- \delta) }
{  \sum_{i=1}^N \1 (\widetilde a_i >1- \delta) \ln (\delta/(1- \widetilde a_i))},
\end{equation}
where $\delta > 0$ is a chosen small threshold and for some $r > 1$, each
\begin{equation}\label{tildea}
\widetilde a_i := \min\{\widehat a_i , 1 - \delta^{r}\}
\end{equation}
is the $\widehat a_i$ of \eqref{hata} truncated at level $1 - \delta^r$ much
closer to 1 than $1-\delta $ in \eqref{best_tilde}.
An obvious reason for the above truncation is that
in general, `noisy' observations in \eqref{hata} need not belong to the interval $(0,1)$ and may exceed 1 in which case the r.h.s.\ of
\eqref{best_tilde} with $\widehat a_i$ instead of $\widetilde a_i$
is undefined.
Even if
$\widehat a_i <1$ as in the  case of the AR(1) estimates in \eqref{han}, the truncation in \eqref{tildea} seem to be necessary
due to the proof of Theorem \ref{thm:2N}. We note a similar truncation of $\widehat a_i$ for technical reasons is used in the parametric context in \cite{ber2010}. On the other hand, our simulations show that when
$r$ is large enough, this truncation has no effect in practice.

\begin{theorem}\label{thm:2N}
Assume (G). As $N \to \infty$, let $\delta=\delta_N \to 0$ so that
\begin{equation}\label{deltaNN}
N \delta^\beta \to \infty  \quad  \text{and} \quad  N \delta^{\beta + 2 \min\{\nu, (r-1)\beta \}} \to 0.
\end{equation}
In addition, let
\begin{eqnarray}\label{condrhoN}
\max_{1\le i \le N} \P(|\widehat \rho_{i}| > \varepsilon) &\le& \frac{\chi}{\varepsilon^p} + \chi', \qquad \forall \varepsilon \in (0,1),
\end{eqnarray}
where $\chi= \chi_{N}, \chi' = \chi'_{N} \to 0$ satisfy
\begin{eqnarray}\label{condN}
\sqrt{ N \delta^{\beta} } \max \Big\{\frac{\chi'}{\delta^\beta}, \Big(\frac{\chi}{\delta^{p+\beta}}\Big)^{1/(p+1)} \Big\} \ln \delta \to 0  \label{cond2N}
\end{eqnarray}
for some $p \ge 1$.
Then
\begin{equation} \label{CLTN}
\sqrt{ N \delta^{\beta} } ( \widetilde \beta_N - \beta) \to_d {\cal N} (0, \beta^2/\kappa).
\end{equation}
\end{theorem}

\section{Estimation of the tail parameter for RCAR(1) panel} \label{sec3}

Let $X_i := \{ X_i (t), \, t \in \Z \}$, $i=1, 2,\dots$, be stationary random-coefficient AR(1) processes in \eqref{RCAR1},
where innovations admit the following decomposition:
\begin{equation}\label{innov}
\zeta_i (t) = b_i \eta(t) + c_i \xi_i (t), \quad t \in \Z, \quad i = 1,2, \dots
\end{equation}
Let the following assumptions hold:
\bigskip

\noi {\bf  (A1)}\ \ $\eta (t)$, $t \in \Z$, are i.i.d.\ with
$\E \eta(t) = 0$, $\E \eta^2(t) = 1$,
$\E |\eta(t)|^{2p} < \infty$ for some $p > 1$.

\bigskip

\noi {\bf (A2)}\ \ $\xi_i (t)$, $t \in \Z$, $i=1,2,\ldots$,  are i.i.d.\ with
$\E \xi_i(t) = 0$, $\E \xi^2_i(t) = 1$,
$\E | \xi_i(t) |^{2p} < \infty$ for the same $p > 1$ as in (A1).

\bigskip

\noi {\bf (A3)}\ \ $(b_i, c_i)$, $i=1,2, \dots$ are i.i.d.\ random vectors with possibly dependent components $b_i  \ge 0$, $c_i \ge 0$ satisfying $\P (b_i + c_i = 0) = 0$ and $\E (b_i^2 + c_i^2) < \infty$.

\bigskip

\noi {\bf (A4)}\ \ $\{ \eta (t), \, t \in \Z \}$, $\{ \xi_i (t), \, t \in \Z \}$, $a_i$ and $(b_i,c_i)$
are mutually independent for each $i=1, 2, \ldots$

\bigskip

Assumptions (A1)--(A4) about the innovations are very general and allow a uniform treatment of
common shock (case $(b_i ,c_i) = (1,0)$) and idiosyncratic shock (case $(b_i,c_i)=(0,1)$) situations.
Similar  assumptions about the innovations are made in \cite{lei2016}.
Under assumptions (A1)--(A4) and (G), there exists a unique strictly stationary solution of \eqref{RCAR1} given by
\begin{eqnarray*} 
 X_i (t) = \sum_{s\le t} a_i^{t-s} \zeta_i (s), \quad t \in \Z,
\end{eqnarray*}
with $\E X_i (t) = 0$ and $\E X_i^2 (t) = \E (b_i^2 + c_i^2) \E (1-a^2_i)^{-1} < \infty$, see  \cite{lei2016}.

From the panel RCAR(1) data $\{ X_i (t), \ t=1, \ldots,T, \ i = 1, \ldots, N\}$
we compute sample lag~1 autocorrelation coefficients
\begin{eqnarray}\label{han}
\widehat a_i := \frac{\sum_{t=1}^{T-1} (X_i(t) - \overline X_i)  (X_i (t+1)- \overline X_i)}{\sum_{t=1}^{T} (X_i (t)- \overline X_i)^2},
\end{eqnarray}
where $\overline X_i := T^{-1} \sum_{t=1}^T X_i (t)$ is the sample mean, $i =1, \ldots, N$.
By the Cauchy-Schwarz inequality, the estimator $\widehat a_i$ in
\eqref{han}  does not exceed 1 in absolute value a.s. Moreover,
$\widehat a_i$ is invariant under the shift and scale transformations of the RCAR(1) process in \eqref{RCAR1}, i.e., we can replace $X_i$ by $\{\sigma_i X_i(t) + \mu_i, \, t \in \Z \}$ with some (unknown) $\mu_i \in \R$ and $\sigma_i > 0$ for every $i =1,2, \ldots$.

To estimate the tail parameter $\beta $ from `noisy' observations $\widehat a_i, i=1, \dots, N$, in \eqref{han} we use the estimator
$\widetilde \beta_N $ in \eqref{best_tilde}.  The crucial `smallness condition' \eqref{condrhoN} on the `noise'
$\widehat \rho_i = \widehat a_i - a_i $ is a consequence of the following result.

\begin{proposition}[\cite{lei2016}]\label{prop:1}   Assume (G) and (A1)--(A4). Then for all $\varepsilon \in (0,1)$ and $T \ge 1$, it holds
	\begin{equation*}
	\P ( |\widehat a_1 - a_1 | > \varepsilon ) \le C (T^{-\min\{p-1,p/2\} } \varepsilon^{-p} + T^{-1} )
	\end{equation*}
	with $C > 0$ independent of $\varepsilon$, $T$.
\end{proposition}

The application of Theorem \ref{thm:2N} leads to the following corollary.

\begin{corollary}\label{thm:2}
Assume (G) and (A1)--(A4). As $N \to \infty$, let $\delta = \delta_N \to 0$ so that
\begin{equation}\label{deltaN}
N \delta^\beta \to \infty  \quad  \text{and} \quad  N \delta^{\beta + 2 \min\{\nu, (r-1)\beta \}} \to 0,
\end{equation}
in addition, let $T = T_N \to \infty$ so that
\begin{eqnarray}\label{cond1}
\sqrt{ N \delta^{\beta} } \gamma \ln \delta \to 0  &&\text{if } \ 1 < p \le 2,\\
\sqrt{ N \delta^{\beta} } \max \Big\{ \frac{1}{T \delta^\beta}, \gamma \Big\} \ln \delta \to 0  &&\text{if } \ 2 < p < \infty,\label{cond2}
\end{eqnarray}
where
\begin{eqnarray}
\gamma := \frac{1}{(T^{\min \{p-1, p/2\}} \delta^{p+\beta})^{1/(p+1)}} \to  0.
\end{eqnarray}
Then
\begin{equation} \label{CLT}
\sqrt{ N \delta^{\beta} } ( \widetilde \beta_N - \beta) \to_d {\cal N} (0, \beta^2/\kappa).
\end{equation}
\end{corollary}

\begin{remark} Condition \eqref{deltaN} restricts the choice of $\delta$ and reduces to that of Theorem \ref{thm:1} with $r$ increasing. In particular, if $\delta = \operatorname{const} N^{-\rm b}$ for some ${\rm b} > 0$
then condition \eqref{deltaN} for $r \ge 2$, $\nu =1 < \beta$
 requires
\begin{equation} \label{bb}
\frac{1}{\beta + 2} < {\rm b} < \frac{1}{\beta}.
\end{equation}	
In view of \eqref{CLT} it makes sense to choose $\rm b$ as large as possible in order to guarantee the fastest convergence rate of the estimator of $\beta$.
Assume $p > 2$ in (A1), (A2). If $\delta = \operatorname{const} N^{-\rm b}$ and $T = N^{\rm a}$ for some ${\rm b} > 0$ satisfying \eqref{bb} and ${\rm a} > 0$, then condition \eqref{cond2} is equivalent to
$$
{\rm a} > \max \Big\{ \frac{1 + {\rm b} \beta}{2}, \frac{1 + {\rm b} \beta}{p} + (2-\beta) {\rm b} + 1 \Big\},
$$
which becomes less restrictive with $p$ increasing and
in the limit $p=\infty $ becomes
\begin{equation} \label{bT2}
{\rm a} > \max \Big\{ \frac{1 + {\rm b} \beta}{2}, (2-\beta) {\rm b} + 1 \Big\}.
\end{equation}
Since for $\beta \in (1,2)$, the lower bound in \eqref{bT2} is $ 1+ (2-\beta){\rm b} > 4/(\beta + 2) > 1$, we conclude that $T$ should grow much faster than $N$.
In general, our results apply to sufficiently long panels.
\end{remark}

Similarly as in the i.i.d.\ case (see \cite{gol1987}), the normalization in \eqref{CLT} can be replaced by
a random quantity expressed in terms of $\widetilde a_i$, $i=1, \dots,N$, alone. That is an actual number of observations usable for inference.

\begin{corollary}\label{cor:1}
Set $\widetilde K_N := \sum_{i=1}^{N} \1 (\widetilde a_i >
1-\delta)$. Under the assumptions of Corollary~\ref{thm:2},
\begin{equation}\label{CLT2}
\sqrt{ \widetilde K_N } (\widetilde \beta_N - \beta) \to_d {\cal N} (0, \beta^2).
\end{equation}
\end{corollary}

The CLTs in \eqref{CLT} and \eqref{CLT2} provide not only consistency of the estimator but also asymptotic confidence intervals for the parameter $\beta$.
The last result can be also used
for testing of long memory in independent RCAR(1) series 
which occurs if $\beta \in (1,2)$.
Note that $\beta=2$ appears as
the boundary between long and short memory.
Indeed, in this case the autocovariance
function of RCAR(1) is not absolutely summable, but the iterated limit of the sample mean of the panel
data follows a normal distribution as for $\beta>2$ (see
\cite{Nedenyi2016}, \cite{pil2014}). Since it is more important to control the
risk of false acceptance of long memory, we choose  
the null hypothesis $H_0 : \beta \ge 2$ vs.\ the
alternative $H_1 : \beta < 2$.
We use the following test statistic
\begin{equation}\label{LMtest}
\widetilde Z_N := \sqrt{\widetilde K_N} (\widetilde \beta_N - 2) / \widetilde \beta_N.
\end{equation}
According to  Corollary \ref{cor:1},  we have
$$\widetilde Z_N \to_d
\begin{cases} {\cal N}(0,1) & \text{if } \beta = 2,\\
 + \infty & \text{if } \beta >  2,\\
 -\infty  & \text{if } \beta < 2.\\
\end{cases}
$$
Fix  $\omega \in (0,1)$ and denote by $z(\omega)$ the
$\omega$-quantile of the standard normal distribution.  The
rejection region $\{\widetilde Z_N < z(\omega)\}$ has asymptotic
level $\omega$ for testing the null hypothesis $H_0 : \beta\geq 2$, and
is consistent against the alternative $H_1 : \beta < 2$.

\section{Simulation study}
\label{sec4}

We examine finite sample performance of the estimator $\widetilde \beta_N$ in \eqref{best_tilde} and the
testing procedure $\widetilde{Z}_N < z(\omega)$ for $H_0 : \beta \ge 2$ at significance level $\omega$.
We compare them with the estimator $\beta_N $ in \eqref{best} and
the test $Z_N := \sqrt{K_N} (\beta_N - 2)/\beta_N  < z(\omega)$, where $K_N := \sum_{i=1}^N \1 (a_i > 1-\delta)$, both based on i.i.d.\ (unobservable) AR coefficients
$a_1, \dots, a_N$.

We consider a panel $\{ X_i (t), \, t= 1,\dots,T, \, i=1,\dots,N \}$, which comprises $N$ independent RCAR(1) series of length $T$. Each of them is generated from i.i.d.\ standard normal innovations $\{ \zeta_i (t) \} \equiv \{ \xi_i (t) \}$ in \eqref{innov} with AR 
coefficient $a_i$ independently drawn from the beta-type density
\begin{equation}\label{betag}
g(x) = \frac{2}{\operatorname{B}(\alpha, \beta)} x^{2\alpha-1} (1-x^2)^{\beta-1}, \quad x \in (0,1),
\end{equation}
with parameters $\alpha > 0$, $\beta > 1$, where $\operatorname{B} (\alpha,\beta) = \Gamma(\alpha) \Gamma (\beta)/\Gamma(\alpha+\beta)$ denotes the beta function.
In this case, the squared
coefficient $a^2_i$ is beta distributed with parameters
$(\alpha,\beta)$. Note \eqref{betag} satisfies \eqref{cond:G} with $\kappa \beta = 2^\beta / \operatorname{B}(\alpha,\beta)$ and $\nu = 1$ if $4\alpha+\beta \neq 3$.
Then RCAR(1) process admits explicit
(unconditional) autocovariance function
\begin{eqnarray} \label{betar}
\E X_i (0) X_i (t) = \E \frac{a_i^{|t|}}{1-a_i^2} = \frac{\operatorname{B}(\alpha+|t|/2,\beta-1)}{\operatorname{B}(\alpha,\beta)}
\sim \frac{\kappa \Gamma(\beta)}{2} t^{-(\beta -1)}, \quad t \to \infty,
\end{eqnarray}
which follows by $\Gamma(t)/\Gamma(t + c) \sim t^{-c}$, $t\to\infty $. The (unconditional) spectral density $f(\lambda)$, $\lambda \in [-\pi,\pi]$, of the RCAR(1) process satisfies
\begin{eqnarray}\label{betaf}
f(\lambda) \ =\ \frac{1}{2\pi} \E |1 -a \e^{-\i \lambda}|^{-2} &\sim&
\kappa_f
\begin{cases}1, &\beta > 2, \\
\ln(1/\lambda), &\beta =2, \\
\lambda^{-(2-\beta)},  &\beta < 2,
\end{cases} \quad \lambda \to 0+,
\end{eqnarray}
where $\kappa_f = (2\pi)^{-1} \E (1-a)^{-2}$ ($\beta > 2$)
and $\kappa_f = \kappa(2\pi)^{-1}$ $(\beta =2)$, $\kappa_f = \kappa \beta(2\pi)^{-1} \int_0^\infty y^{\beta-1} (1+y^2)^{-1} \d y$ $(1<\beta<2)$ (see \cite{lei2014}).
From \eqref{betar}, \eqref{betaf} we see that (unconditionally) $X_i$ behaves as $I(0)$ process for $\beta > 2$ and as $I(d)$ process for $\beta \in (1,2)$ with fractional integration parameter $d =  1- \beta/2 \in (0, 1/2)$.  Particularly,
$\beta = 1.5 $ corresponds to $d = 0.25 $ (the middle point on the interval $(0,1/2)$), whereas $\beta = 1.75$ to
$d = 0.125$. Increasing  parameter $\alpha $ `pushes' the distribution of the AR coefficient towards $x=1$, see Figure \ref{fig:3} [left], and affects the asymptotic constants of $g(x)$ as $x \to1-$.
A somewhat unexpected feature of this model is a considerable amount of `spurious' long memory for $\beta > 2 $. Figure \ref{fig:3} [right] shows
the graph of the spectral density in \eqref{betaf} which is bounded though sharply increases at the origin for $\beta = 2.5 $, $\alpha \ge 1.5$. One may
expect that most time series tests 
applied to a (Gaussian) process with a spectral density as the one in Figure \ref{fig:3} [right]
for $\beta = \alpha = 2.5$
will incorrectly reject the short memory hypothesis in favour of long memory. See Remark 2.
\begin{figure}[h!]
\begin{tabular}{cr}
\includegraphics[width=0.4\textwidth]{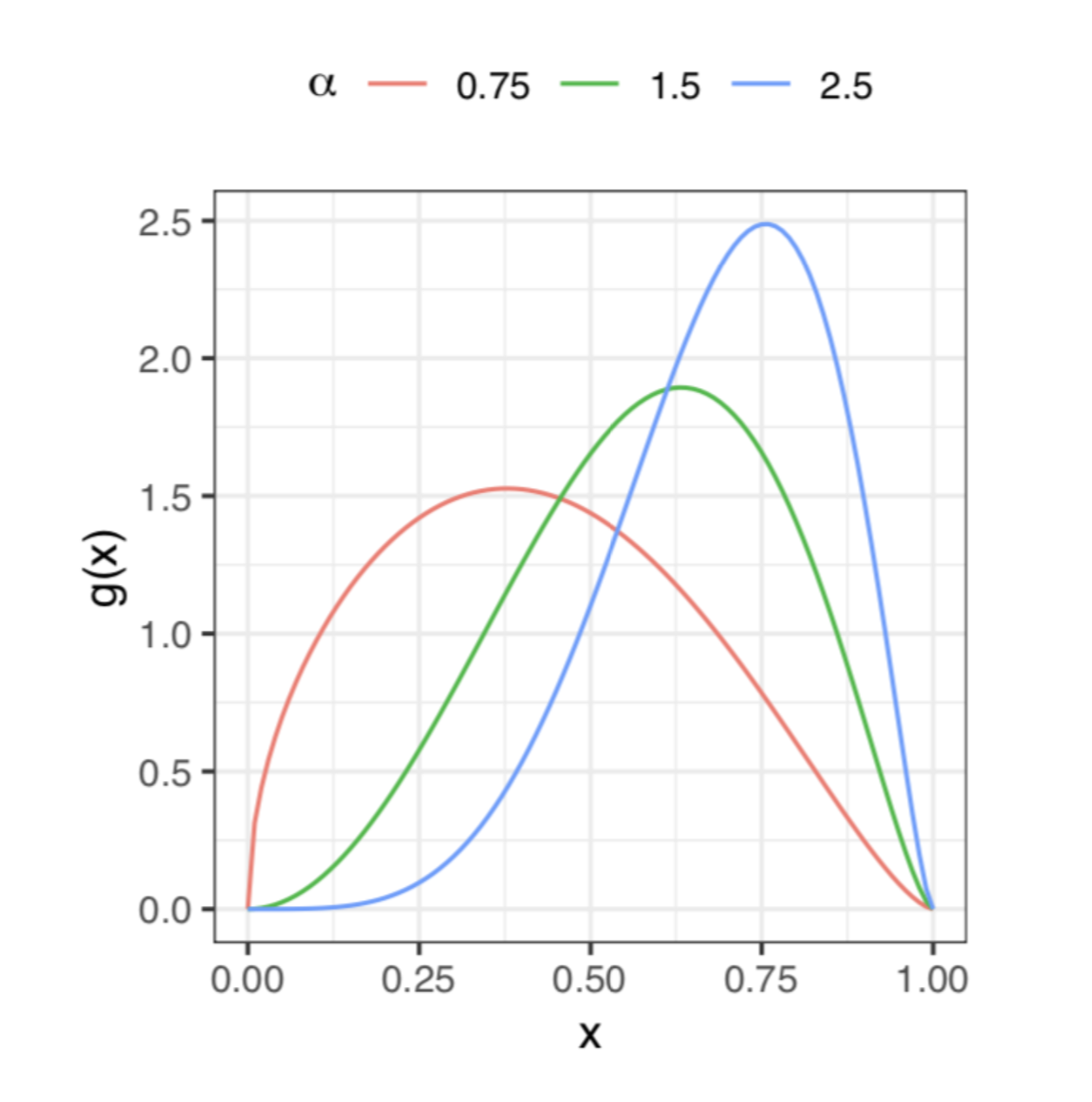}	
&\includegraphics[width=0.4\textwidth]{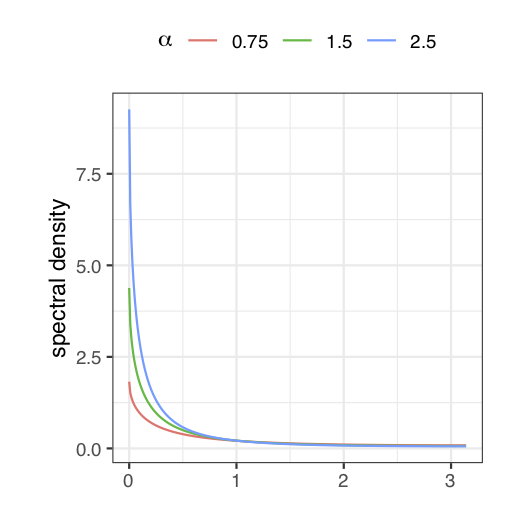}
\end{tabular}
	\caption{[left] Probability density $g(x)$, $x \in (0,1)$, in \eqref{betag} for $\beta = 2.5$.
[right]  Spectral density $f(\lambda)$, $\lambda \in [0,\pi]$, in \eqref{betaf}  
for the same value of $\beta$. 	
}
\label{fig:3}
\end{figure}

Let us turn to the description of our simulation procedure.
We simulate 5000 panels for each configuration of $N$, $T$,
$\alpha$ and $\beta$, where
\begin{itemize}
	\item  $(N,T) = (750, 1000), \, (750,2000)$,
	\item   $\beta = 1.5, \, 1.75, \, 2, \, 2.25,\, 2.5$,
	\item $\alpha = 0.75,\, 1.5,\, 2.5$.
\end{itemize}
As usual in tail-index estimation, the most difficult and delicate task is choosing the threshold.
We note that conditions in Theorem  \ref{thm:1} and Corollary \ref{thm:2} hold asymptotically
and allow for different choice of $\delta$; moreover, they depend on (unknown) $\beta $ and the second-order parameter $\nu$.
Roughly speaking, larger $\delta $ increases the number of the usable observations (upper order statistics) in \eqref{best_tilde} and \eqref{best}, hence
makes standard deviation of the estimator smaller, but at the same time
increases bias since the density $g(x)$ in \eqref{betacond} is more likely to deviate from its  asymptotic form on a longer interval
$(1-\delta, 1)$.  In the i.i.d.\ case or $\beta_N$, the `optimal'  choice of $\delta $ is given by
\begin{equation}\label{deltastar}
\delta^* :=  \Big( \frac{\beta (\beta + \nu)^2}{2 \tau^2 \nu^3 \kappa N} \Big)^{1/(\beta+2\nu)},
\end{equation}
see equation (4.3.8) in \cite{gol1987}, which minimizes the asymptotic
mean squared error of $\beta_N$ provided
the distribution of $a_i$ satisfies a generally stronger version of the second-order condition in \eqref{cond:G}:
\begin{equation}\label{G2}
\P(a_i > 1-x) = \kappa x^\beta (1+ \tau x^\nu + o (x^\nu)),  \quad x \to 0+,
\end{equation}
for the same $\beta > 1$ and some parameters $\nu > 0$, $\kappa > 0$, $\tau \neq 0$.
Then on average the computation of $\beta_N$ uses
\begin{equation}\label{kstar}
\E \sum_{i=1}^N \1 (a_i > 1- \delta^* ) \sim \Big( \frac{(1-\rho) N^{-\rho}}{B \sqrt{-2\rho}} \Big)^{2/(1-2\rho)} =: k^*
\end{equation}
upper order statistics 
of $a_1,\dots,a_N$, where the second-order parameters $\rho := - \nu/\beta < 0$, $B := (\nu/\beta) \kappa^{-\nu/\beta} \blu{\tau} \neq 0$ are more convenient to estimate, see e.g.\ \cite{pau2017}. Therefore, given the order statistics $a_{(1)} \le \dots \le a_{(N)}$, we use (random) $\delta = 1- a_{(N - \lfloor k^*\rfloor)}$ 
as a substitute for $\delta^*$. Furthermore, since $\delta^*$ of \eqref{deltastar} yields asymptotic normality of $\beta_N$ in \eqref{best} with non-zero mean,
we choose a smaller sample fraction $(k^*)^\epsilon < k^*$ with $\epsilon \in (0,1)$ and the corresponding
\begin{equation} \label{delta1}
\delta = 1-a_{(N-\lfloor(k^*)^\epsilon\rfloor)},
 \end{equation}
for which the asymptotic normality of $\beta_N$ holds as in Theorem \ref{thm:1}. In our simulations of $\beta_N$ in \eqref{best} we use $\delta$ in \eqref{delta1} with several values of $\epsilon \in (0,1)$ and $k^*$  is obtained by replacing $\rho$, $B$ in \eqref{kstar} by their semiparametric estimates, see \cite{fra2003}, \cite{gom2002}. We calculate the latter estimates from $a_1, \dots, a_N$ using the algorithm in \cite{gom2009}.
Because of the lack of the explicit formula minimizing the mean squared error of
$\widetilde \beta_N$, in our simulations of the latter estimator  we  use a similar threshold $\delta \in (0,1)$, viz.,
\begin{equation}\label{delta}
\delta = 1 - \widehat a_{(N- \lfloor(\widehat k^*)^\epsilon\rfloor)},
\end{equation}
where $\widehat{a}_{(1)} \le \dots \le \widehat{a}_{(N)}$ denote the order statistics calculated from the simulated
RCAR(1) panel,
and $\widehat k^*$ is the analogue of $k^*$ computed from $\widehat a_1, \dots, \widehat a_N$.
Moreover, for our simulations of $\widetilde \beta_N$ we use $r = 10$ in \eqref{tildea}, though any large $r$ could be chosen.

Table \ref{table:mse2} illustrates the effect of $\epsilon$ in \eqref{delta1}, \eqref{delta} on the performance of $\beta_N$, $\widetilde \beta_N$, respectively, for $(N,T) = (750, 1000)$.
Choosing smaller $\epsilon$, the bias of
the estimators decreases in most cases, whereas their standard deviation increases.
The choice $\epsilon = 0.9$ seems to be near `optimal' in the sense of RMSE.

{\renewcommand{\arraystretch}{0.4}
	
	\begin{table}[ht]
		\centering
		\begin{tabular}{lrrr rrrrr rrrr}
			\hline\\
			\multicolumn{2}{l}{}& \multicolumn{3}{c}{$\beta=1.5$}& & \multicolumn{3}{c}{$\beta=2$}& & \multicolumn{3}{c}{$\beta=2.5$}\\ \\
			\cline {3-5} \cline{7-9} \cline{11-13}\\
			$\epsilon$ &$\alpha=$& 0.75& 1.5& 2.5& &  0.75& 1.5& 2.5& &  0.75& 1.5& 2.5\\ \\
			\hline\\
			
			&&\multicolumn{11}{c}{RMSE of $\widetilde \beta_N$} \\ \\
			1&& 0.18 & 0.12 & 0.11 && 0.22 & 0.24 & 0.24 && 0.36 & 0.40 & 0.42 \\ \\
			0.9&& 0.18 & 0.15 & 0.18 && 0.19 &  0.21 & 0.20 && 0.29 & 0.33 & 0.34 \\ \\
			0.8&& 0.17 & 0.23 & 0.30 && 0.21 &  0.24 & 0.25 && 0.29 & 0.33 & 0.33 \\ \\
			0.7&& 0.25 & 0.35 & 0.47 && 0.29 &  0.33 & 0.36 && 0.36 & 0.40 & 0.41 \\ \\
			
			&&\multicolumn{11}{c}{RMSE of $\beta_N$} \\ \\
			1&& 0.15 &  0.18 &  0.19 &&  0.26 &  0.29 &  0.31 &&  0.38 &  0.43 &  0.47 \\ \\
			0.9&& 0.13 &  0.16 &  0.16 &&  0.21 &  0.25 &  0.26 &&  0.31 &  0.36 &  0.39 \\ \\
			0.8&& 0.16 &  0.18 &  0.19 &&  0.22 &  0.25 &  0.27 &&  0.31 &  0.35 &  0.37 \\ \\
			0.7&& 0.21 &  0.24 &  0.24 &&  0.28 &  0.31 &  0.33 &&  0.36 &  0.41 &  0.42 \\ \\ \\
			
			&&\multicolumn{11}{c}{Bias of $\widetilde \beta_N$} \\ \\
			1&& -0.07 & -0.05 & -0.01 && -0.19 & -0.20 & -0.20 && -0.32 & -0.36 & -0.38\\ \\
			0.9&& -0.01 &  0.04 &  0.10 && -0.11 & -0.10 & -0.08 && -0.21 & -0.25 & -0.25 \\ \\
			0.8&& 0.05 &  0.13 &  0.22 && -0.04 & -0.01 &  0.02 && -0.13 & -0.15 & -0.13 \\ \\
			0.7&& 0.11 &  0.23 &  0.36 &&  0.02 &  0.07 &  0.12 && -0.05 & -0.06 & -0.03 \\ \\
			
			&&\multicolumn{11}{c}{Bias of $\beta_N$} \\ \\
			1&& -0.12 & -0.14 & -0.15 && -0.23 & -0.26 & -0.28 && -0.35 & -0.40 & -0.44 \\ \\
			0.9&& -0.06 & -0.08 & -0.09 && -0.15 & -0.17 & -0.19 && -0.24 & -0.29 & -0.32 \\ \\
			0.8&& -0.03 & -0.04 & -0.04 && -0.09 & -0.11 & -0.12 && -0.17 & -0.21 & -0.22 \\ \\
			0.7&& 0.00 &  0.00 &  0.00 && -0.05 & -0.05 & -0.06 && -0.10 & -0.13 & -0.14 \\ \\
		\end{tabular}
		\caption{Performance of $\widetilde \beta_N$, $\beta_N$ for $(N,T)=(750,1000)$, $a_i^2 \sim \operatorname{Beta}(\alpha, \beta)$, using $\delta$ in \eqref{delta}, \eqref{delta1}, respectively, with estimated parameters $B, \rho$ and $r=10$. The number of replications is $5000$.}
		\label{table:mse2}
\end{table}}

Table \ref{table:mse1} presents the performance of $\widetilde \beta_N$ and $\beta_N$ with $\epsilon = 0.9$,
for a wider choice of parameters $\alpha, \beta $ and two values of $T$.
We see that the sample RMSE 
of both statistics $\widetilde \beta_N$ and $\beta_N$ are very similar almost uniformly in $\alpha, \beta, T$
(the only exception seems the case $\alpha = 2.5, \beta= 1.25, T = 1000$).
Surprisingly, in most cases the statistic $\widetilde \beta_N$ for $T=1000$ seems to be more accurate than the same statistic for $T = 2000$ and the `i.i.d.' statistic $\beta_N$. This unexpected effect can be explained by
a positive bias introduced by estimated $a_i$ for $\epsilon = 0.9$  which partly compensates the negative bias of $\beta_N$, see Table \ref{table:mse1}.

{\renewcommand{\arraystretch}{0.4}
		
	\begin{table}[ht]
		\centering
		\begin{tabular}{lrrr rlrrr lrrr}
			\hline\\
			\multicolumn{2}{l}{}& \multicolumn{3}{c}{$\beta=1.25$}& & \multicolumn{3}{c}{$\beta=1.5$}& & \multicolumn{3}{c}{$\beta=1.75$}\\ \\
			\cline {3-5} \cline{7-9} \cline{11-13}\\
			&$\alpha=$& 0.75& 1.5& 2.5& &  0.75& 1.5& 2.5& &  0.75& 1.5& 2.5\\ \\
			\hline\\
			&&\multicolumn{11}{c}{RMSE} \\ \\
			$\widetilde \beta_N, \, T = 1000$ && 0.11 &  0.17 &  0.27 &&  0.18 &  0.15 &  0.18 &&  0.15 &  0.16 &  0.17 \\ \\
			$\widetilde \beta_N, \, T = 2000$ && 0.10 &  0.12 &  0.15 &&  0.12 &  0.14 &  0.14 &&  0.16 &  0.17 &  0.17 \\ \\ \\
			$\beta_N$  && 0.10 &  0.12 &  0.13 &&  0.13 &  0.16 &  0.16 &&  0.17 &  0.20 &  0.21 \\ \\ \\
			
			&&\multicolumn{11}{c}{Bias} \\ \\
			$\widetilde \beta_N, \, T = 1000$ && 0.04 &  0.12 &  0.23 && -0.01 &  0.04 &  0.10 && -0.05 & -0.03 &  0.00 \\ \\
			$\widetilde \beta_N, \, T = 2000$ && 0.00 &  0.04 &  0.09 && -0.04 & -0.02 &  0.01 && -0.08 & -0.08 & -0.07 \\ \\ \\
			$\beta_N$ && -0.04 & -0.05 & -0.06 && -0.06 & -0.08 & -0.09 && -0.10 & -0.12 & -0.14 \\ \\ \\
			
			\hline\\
			\multicolumn{2}{l}{}& \multicolumn{3}{c}{$\beta=2$}& & \multicolumn{3}{c}{$\beta=2.25$}& & \multicolumn{3}{c}{$\beta=2.5$}\\ \\
			\cline {3-5} \cline{7-9} \cline{11-13}\\
			&$\alpha=$& 0.75& 1.5& 2.5& &  0.75& 1.5& 2.5& &  0.75& 1.5& 2.5\\ \\
			\hline\\
			&&\multicolumn{11}{c}{RMSE} \\ \\
			$\widetilde \beta_N, \, T = 1000$ && 0.19 &  0.21 &  0.20 &&  0.23 &  0.26 &  0.26 &&  0.29 &  0.33 &  0.34 \\ \\
			$\widetilde \beta_N, \, T = 2000$ && 0.20 &  0.22 &  0.23 &&  0.24 &  0.28 &  0.29 &&  0.30 &  0.34 &  0.36 \\ \\ \\
			$\beta_N$ && 0.21 &  0.25 &  0.26 &&  0.26 &  0.30 &  0.32 &&  0.31 &  0.36 &  0.39 \\ \\ \\
			
			&&\multicolumn{11}{c}{Bias} \\ \\
			$\widetilde \beta_N, \, T = 1000$ && -0.11 & -0.10 & -0.08 && -0.16 & -0.17 & -0.17 && -0.21 & -0.25 & -0.25 \\ \\
			$\widetilde \beta_N, \, T = 2000$ && -0.12 & -0.14 & -0.14 && -0.17 & -0.20 & -0.21 && -0.23 & -0.27 & -0.28 \\ \\ \\
			$\beta_N$ && -0.15 & -0.17 & -0.19 && -0.19 & -0.23 & -0.25 && -0.24 & -0.29 & -0.32 \\ \\
		\end{tabular}	
		\caption{Performance of $\widetilde \beta_N$, $\beta_N$
			for $(N,T)=(750,1000)$ and $(N,T)=(750,2000)$, $a_i^2 \sim \operatorname{Beta}(\alpha, \beta)$, using $\delta$ in \eqref{delta}, \eqref{delta1}, respectively, with estimated parameters $B, \rho$ and $\epsilon=0.9$, $r=10$. The number of replications is $5000$.}
		\label{table:mse1}
\end{table}}

Next, we examine the performance of the test statistics $\widetilde{Z}_N$ and $Z_N$ using $\delta$ in \eqref{delta} and \eqref{delta1}, respectively.
Tables \ref{table3}, \ref{table4} reports rejection rates of $H_0 : \beta \ge 2$ in favour of $H_1 : \beta < 2$ at level $\omega = 5\%$ using $\widetilde{Z}_N$ and $Z_N$ for $(N,T) = (750,2000)$ and different values of $\alpha,\beta$ and $\epsilon$, $r$. The results are almost the same when using $\widetilde{Z}_N$ for $r=3$ and $r=10$.
Table \ref{table3} shows that choosing $\epsilon=0.7$ for $\beta \ge  2.25$ and all values of $\alpha$, the incorrect rejection rates of the null  in favour of the long memory alternative are much smaller than $5\%$, despite the spurious long memory in Figure \ref{fig:3} [right].
Choosing $\epsilon=0.9$, they increase a bit but are still smaller than 5\% using $\widetilde Z_N$, see Table~\ref{table4}.
However, at the boundary $\beta = 2$ between short and long memory, the empirical size of the tests is not well observed. The deviation from the nominal level is especially noticeable in the case of the `i.i.d.'  statistic $Z_N$. This size distortion may be explained by the fact that the tails of the empirical distribution of $Z_N$ and $\widetilde Z_N$ are not well-approximated by tails of the limiting normal distribution. More extensive simulations of
the performance of $\widetilde Z_N$ and $Z_N$ for other choices of $\epsilon$, $r$, $N$, $T$ are presented in
the arXiv version \cite{lei2018} of this paper.

{\renewcommand{\arraystretch}{0.4}
	
	\begin{table}[ht]
		\centering
		\begin{tabular}{ lrrrr rrrrr rrrr rrrr rrrr}
			\hline\\
			\multicolumn{2}{l}{} & \multicolumn{3}{c}{$\beta=1.25$}& & \multicolumn{3}{c}{$\beta=1.5$}& & \multicolumn{3}{c}{$\beta=1.75$}\\ \\
			\cline {3-5} \cline{7-9} \cline{11-13}\\
			& $\alpha = $ &0.75& 1.5& 2.5& &  0.75& 1.5& 2.5&&  0.75& 1.5& 2.5\\ \\
			\hline\\
			$\widetilde Z_N$, $r=3$ && 93.5 & 76.4 & 56.2 && 67.1 & 50.7 & 37.6 && 29.1 & 23.2 & 18.6\\ \\
			$\widetilde Z_N$, $r=10$ && 93.5 & 76.4 & 56.2 && 67.1 & 50.7 & 37.6 && 29.2 & 23.2 & 18.6\\ \\
			$Z_N$ && 97.0 & 94.0 & 93.5 && 76.8 & 69.6 & 68.7 && 36.7 & 35.7 & 35.8\\ \\ \\			
			\hline\\
			\multicolumn{2}{l}{} & \multicolumn{3}{c}{$\beta=2$}& & \multicolumn{3}{c}{$\beta=2.25$}& & \multicolumn{3}{c}{$\beta=2.5$}\\ \\
			\cline {3-5} \cline{7-9} \cline{11-13}\\
			& $\alpha = $ &0.75& 1.5& 2.5& &  0.75& 1.5& 2.5&&  0.75& 1.5& 2.5\\ \\
			\hline\\
			$\widetilde Z_N$, $r=3$ && 7.8 &  7.9 &  6.1 &&  0.8 &  1.5 &  1.8 &&  0.1 &  0.2 &  0.4 \\ \\
			$\widetilde Z_N$, $r=10$ && 8.0 &  7.9 &  6.1 &&  0.9 &  1.5 &  1.8 &&  0.1 &  0.2 &  0.4\\ \\			
			$Z_N$ && 10.8 & 11.6 & 12.3 &&  1.7 &  2.6 &  3.0 &&  0.2 &  0.2 &  0.6\\ \\
		\end{tabular}
		\caption{
			Rejection rates (in \%) of $H_0 : \beta\geq 2$ at level $\omega=5\%$ with $\widetilde Z_N$, $Z_N$ for $(N, T) = (750,2000)$, $a_i^2 \sim {\operatorname{Beta}}(\alpha,\beta)$, using $\delta$ in \eqref{delta}, \eqref{delta1}, respectively, with estimated parameters $B$, $\rho$ and $\epsilon =0.7$.
			The number of replications is $5000$.
		}
		\label{table3}
	\end{table}}

	{\renewcommand{\arraystretch}{0.4}
				\begin{table}[ht]
			\centering
			\begin{tabular}{ lrrrr rrrrr rrrr rrrr rrrr}
				\hline\\
				\multicolumn{2}{l}{} & \multicolumn{3}{c}{$\beta=1.25$}& & \multicolumn{3}{c}{$\beta=1.5$}& & \multicolumn{3}{c}{$\beta=1.75$}\\ \\
				\cline {3-5} \cline{7-9} \cline{11-13}\\
				& $\alpha = $ &0.75& 1.5& 2.5& &  0.75& 1.5& 2.5&&  0.75& 1.5& 2.5\\ \\
				\hline\\
				$\widetilde Z_N$, $r=3$ && 100.0 &  99.9 &  99.3 &&  98.5 &  95.3 &  91.9 &&  72.9 &  68.3 &  62.9\\ \\
				$\widetilde Z_N$, $r=10$ && 100.0 &  99.9 &  99.3 &&  98.7 &  95.3 &  91.9 &&  76.1 &  68.4 &  62.9\\ \\
				$Z_N$ && 100.0 &  99.9 &  99.9 &&  99.2 &  97.9 &  97.6 &&  81.0 &  78.1 &  78.1\\ \\ \\		
				\hline\\
				\multicolumn{2}{l}{} & \multicolumn{3}{c}{$\beta=2$}& & \multicolumn{3}{c}{$\beta=2.25$}& & \multicolumn{3}{c}{$\beta=2.5$}\\ \\
				\cline {3-5} \cline{7-9} \cline{11-13}\\
				& $\alpha = $ &0.75& 1.5& 2.5& &  0.75& 1.5& 2.5&&  0.75& 1.5& 2.5\\ \\
				\hline\\
				$\widetilde Z_N$, $r=3$ && 19.8 & 25.1 & 23.8 &&  1.1 &  4.2 &  4.6 &&  0.0 &  0.3 &  0.6 \\ \\
				$\widetilde Z_N$, $r=10$ && 26.7 & 25.5 & 23.8 &&  2.2 &  4.4 &  4.6 &&  0.1 &  0.3 &  0.6\\ \\			
				$Z_N$ && 32.2 & 34.1 & 37.0 &&  3.6 &  5.8 &  8.0 &&  0.1 &  0.5 &  1.2\\ \\
			\end{tabular}
			\caption{
				Rejection rates (in \%) of $H_0 : \beta\geq 2$ at level $\omega=5\%$ with $\widetilde Z_N$, $Z_N$ for $(N, T) = (750,2000)$, $a_i^2 \sim {\operatorname{Beta}}(\alpha,\beta)$, using $\delta$ in \eqref{delta}, \eqref{delta1}, respectively, with estimated parameters $B$, $\rho$ and $\epsilon =0.9$.
				The number of replications is $5000$.
			}
			\label{table4}
		\end{table}	}


\medskip

\begin{remark} In time series theory, several semi-parametric tests for long memory were developed,
see \cite{gir2003}, \cite{grom2018}, \cite{lob1998}. Clearly, these tests cannot be
applied to individual RCAR(1) series, the latter being always short memory a.s., independently of the value of $\beta $ and the distribution
of the AR coefficient $a_i$. However, in practice one can apply the above-mentioned tests to the aggregated RCAR(1)
series $\{\bar X_N(1), \dots, \bar X_N (T)\} $
in \eqref{barXN} whose autocovariance decays as $t^{-(\beta-1)}$, $t \to \infty$,  see \eqref{covX}. In \cite{lei2018} we report a Monte Carlo analysis of the finite sample performance of the V/S test (see \cite{gir2003})
applied to the aggregated RCAR(1) series with short memory $(\beta = 2.5)$ for the same model as above. Since the V/S statistic is quite
sensitive to the choice of the tuning parameter, \cite{lei2018} derived its data-driven choice by
expanding the HAC estimator as proved by \cite{abad09a} and minimizing its mean squared error under the null hypothesis. The simulations in \cite{lei2018} show that the V/S test is not valid,
in the sense that its empirical size is not close to the nominal level.  The reason why the V/S test fails for our panel model may be due to the presence  of the spurious long memory  (see Figure \ref{fig:3}).
\end{remark}

\clearpage

\section{Proofs}\label{sec5}

{\it Notation.}
In what follows, let
$G_N(x) := N^{-1} \sum_{i=1}^N \1 (a_i \le x)$ and $\widehat G_N (x) := N^{-1} \sum_{i=1}^N \1 (\widehat a_i \le x)$, where $\widehat{a}_1, \dots, \widehat{a}_N$ are defined by \eqref{hata} and $a_1, \dots, a_N$ are i.i.d.\ with $G (x) := \P(a_1 \le x)$, $x \in \R$.
\medskip

\begin{proof}[Proof of Theorem~\ref{thm:1}]
	We rewrite the estimator in \eqref{best} as
	$$
	\beta_N =\frac{1-G_N (1-\delta) }{\int_{1-\delta}^1 \ln (\delta/(1- x)) \d G_N (x)} = \frac{1-G_N (1-\delta) }{\int_{1-\delta}^1 (1-G_N (x)) \frac{\d x}{1- x}} = \frac{1-G_N (1-\delta)}{\int_0^\delta (1-G_N (1- x)) \frac{\d x}{x}}.
	$$
	Next, we decompose $\beta_N - \beta = D^{-1} \sum_{i=1}^{4} I_i$,
	where
	\begin{eqnarray}\label{def:Ii}
	I_1 &:=& \beta \int_0^\delta (G_N(1-x) - G(1-x)) \frac{\d x}{x},\qquad
	I_2 := - (G_N(1-\delta) - G(1-\delta)),\\
	I_3 &:=& - \beta \int_0^\delta (1-\kappa x^\beta - G(1-x)) \frac{\d x}{x},\qquad
	I_4 := 1 - \kappa \delta^\beta - G(1-\delta)\nn
	\end{eqnarray}
	and
	\begin{eqnarray}\label{def:D}
	D := \int_0^\delta (1 - G_N(1-x)) \frac{\d x}{x} = \frac{1}{\beta} ( \kappa \delta^\beta-I_1-I_3 ).
	\end{eqnarray}
	According to the assumptions $(N \delta^{\beta})^{1/2} \delta^{\nu} \to 0$ and (G), we get
	$(N \delta^{-\beta})^{1/2} I_4 \to 0$ and $(N \delta^{-\beta})^{1/2} I_3 \to 0$.
	
	From the tail empirical process theory, see e.g.\ Theorem~1 in \cite{ein1990}, (1.1)--(1.3) in \cite{mas1988}, we have that
	\begin{eqnarray}\label{cltG}
	(N \delta^{-\beta})^{1/2} ( G_N (1- x\delta) - G(1- x\delta) ) \to_{D[0,1]} \kappa^{1/2} B(x^{\beta}),
	\end{eqnarray}
	where $\{ B(x), \ x \in [0,1] \}$ is a standard Brownian motion.
	Therefore, we can expect that
	\begin{eqnarray}\label{main}
	(N \delta^{-\beta})^{1/2} (I_1 + I_2) \to_d
	\kappa^{1/2} \Big( \beta \int_0^1  B(x^{\beta}) \frac{\d x}{x}
	-  B(1) \Big).
	\end{eqnarray}
	The main technical point to prove
	\eqref{main} is to justify the application of the invariance principle \eqref{cltG}
	to the integral $(N \delta^{-\beta})^{1/2} I_1$,
	which is not a continuous functional in the uniform topology on the whole space $D[0,1]$. For $\varepsilon >0$, we split $I_1 := \beta (I_0^\varepsilon + I_\varepsilon^1)$, where
	$$
	I_0^\varepsilon := \int_0^\varepsilon (G_N(1 - \delta x)- G(1-\delta x)) \frac{\d x}{x}, \quad I_\varepsilon^1  := \int_\varepsilon^1  (G_N(1 - \delta x)- G(1-\delta x))\frac{\d x}{x}.
	$$
	By \eqref{cltG},
	$  (N \delta^{-\beta})^{1/2} I_\varepsilon^1  \to_d  \kappa^{1/2} \int_\varepsilon^1  B(x^{\beta}) \frac{\d x}{x}$, where $\E |\int_\varepsilon^1  B(x^{\beta}) \frac{\d x}{x} - \int_0^1  B(x^{\beta}) \frac{\d x}{x} |^2
	\to 0$ as $\varepsilon \to  0$. Hence,  \eqref{main} follows from
	\begin{eqnarray}\label{Gep}
	\lim_{\varepsilon \to 0} \limsup_{N\to  \infty}  \E | (N \delta^{-\beta})^{1/2} I_0^\varepsilon |^2 = 0.
	\end{eqnarray}
	In the i.i.d.\ case $\E | I_0^\varepsilon |^2
	=
	\int_0^\varepsilon \int_0^\varepsilon \operatorname{Cov}(G_N(1 - \delta x), G_N (1-\delta y)) \frac{\d x \d y}{xy}$, where
	\begin{eqnarray*}
		\operatorname{Cov}(G_N(x), G_N(y))
		=N^{-1} G(x \wedge y) (1 - G(x \vee y))\le N^{-1} (1 - G(x \vee y)),
	\end{eqnarray*}
	and
	\begin{eqnarray}\label{varI10}
	\E |I_0^\varepsilon|^2 \le \frac{C}{N} \int_0^\varepsilon \frac{\d x}{x} \int_0^x (1- G(1 - \delta y)) \frac{\d y}{y} \le \frac{C}{N} \int_0^\varepsilon \frac{\d x}{x} \int_0^x (\delta y)^\beta \frac{\d y}{y} = \frac{C}{N \delta^{-\beta}} \int_0^\varepsilon x^{\beta -1} \d x = \frac{C \varepsilon^\beta}{N \delta^{-\beta}},
	\end{eqnarray}
	proving \eqref{Gep} and hence  \eqref{main} too.
	
	Finally, we obtain $\delta^{-\beta} D \to_p \kappa/\beta$ in view of $(N\delta^{-\beta})^{1/2} (I_1 + I_3) = O_p(1)$ and $N\delta^\beta \to \infty.$
	
	We conclude that
	\begin{eqnarray}\label{betalim}
	(N \delta^{\beta})^{1/2} (\beta_N - \beta) \to_d \frac{\beta}{\kappa^{1/2}}
	\Big(\beta  \int_0^1  B(x^{\beta}) \frac{\d x}{x}
	-   B(1)\Big) =: W.
	\end{eqnarray}
	Clearly, $W$ follows a normal distribution with zero mean and variance
	\begin{eqnarray*}
		\E W^2 = \frac{\beta^2}{\kappa} \Big(2 \beta^2 \int_0^1 \frac{\d x}{x} \int_0^x y^{\beta -1} \d y
		- 2 \beta \int_0^1 x^{\beta-1} \d x + 1 \Big) = \frac{\beta^2}{\kappa},
	\end{eqnarray*}
	which agrees with the one in \cite{gol1987}. The proof is complete.
\end{proof}

In the proof of Theorem \ref{thm:2N} we will use the following proposition.

\begin{proposition}\label{prop:2}
	Assume (G). 
	As $N \to \infty$, let $\delta = \delta_N \to 0$ so that $N \delta^\beta \to \infty$ and \eqref{condrhoN}, \eqref{condN} hold.
	Then
	\begin{align}\label{rmd2}
	(N \delta^{-\beta})^{1/2} &(\widehat G_N (1-\delta) - G_N(1-\delta)) = o_{p}(1),\\
	(N \delta^{-\beta})^{1/2} \int_{\delta^r}^\delta &( \widehat G_N(1-x) - G_N (1-x)) \frac{\d x}{x} = o_{p} (1).\label{rmd1}
	\end{align}
\end{proposition}

\begin{proof}
	For $x \in [1-\delta, 1]$, write
	\begin{eqnarray*}
		\widehat G_N (x) - G_N(x) = \frac{1}{N} \sum_{i=1}^N
		(\1 (a_i + \widehat \rho_i \le x) - \1(a_i \le x)) = D'_N(x) - D''_N (x),
	\end{eqnarray*}
	where $\widehat \rho_{i}:= \widehat a_i - a_i$, $i=1,\ldots,N$, and
	\begin{eqnarray*}
		D'_N(x) &:=& \frac{1}{N} \sum_{i=1}^N \1(x < a_i \le x - \widehat \rho_{i},\ \widehat \rho_i \le 0), \\
		D''_N (x) &:=& \frac{1}{N} \sum_{i=1}^N \1( x - \widehat \rho_i < a_i \le x, \ \widehat \rho_{i} > 0).
	\end{eqnarray*}
	For all $\gamma > 0$,
	$$
	0 \le D'_N(x) \le \frac{1}{N} \sum_{i=1}^N \1(x < a_i \le x + \gamma \delta) + \frac{1}{N}  \sum_{i=1}^N \1 ( |\widehat \rho_i|>\gamma \delta) =:
	I'_{N}(x) + I''_{N},
	$$
	where by 
\eqref{condrhoN}
\begin{equation}\label{IN}
	\E I''_{N}  \le \max_{1\le i \le N} \P ( | \widehat \rho_{i} | > \gamma \delta ) \le
 \frac{\chi}{(\gamma \delta)^p} + \chi'
	\end{equation}
and
	\begin{equation}\label{ineq:I1}
	\E I'_N (x) = \P ( x < a_1 \le x + \gamma \delta ) \le C \int_{x}^{x+\gamma \delta} (1-u)^{\beta-1} \d u \le 
	C \gamma \delta^{\beta}
	\end{equation}
	holds uniformly for all $x \in [1-\delta,1]$ according  to \eqref{cond:G}.
Choose
\begin{equation}
\gamma := \big(\frac{\chi}{\delta^{p+\beta}}\big)^{1/(p+1)},
\end{equation}
then $\chi/(\gamma \delta)^p \sim \gamma \delta^\beta $ and the r.h.s.\ of \eqref{IN} does not exceed
$C (\gamma \delta^\beta + \chi') \le C \max\{\gamma \delta^\beta, \chi'\}$.  	
Under the conditions \eqref{condrhoN}, \eqref{condN}, 
from \eqref{IN}, \eqref{ineq:I1} it follows that
	$$
	(N\delta^{-\beta})^{1/2} \int_{\delta^r}^{\delta} \E D'_N (1-x) \frac{\d x}{x} \le C |\ln \delta| (N\delta^{-\beta})^{1/2} \big( \E I''_N + \sup_{x \in [0,\delta]} \E I'_N (1-x) \big) = o (1),
	$$
	hence
	$$
	(N\delta^{-\beta})^{1/2} \int_{\delta^r}^{\delta} D'_N (1-x) \frac{\d x}{x} = o_p (1)
	$$
	by Markov's inequality. Since
	$$
	(N\delta^{-\beta})^{1/2} \int_{\delta^r}^{\delta} D''_N (1-x) \frac{\d x}{x} = o_p (1)
	$$
	is analogous, this proves
	\eqref{rmd1}. The same proof works for the relation \eqref{rmd2}.
\end{proof}

\begin{proof}[Proof of Theorem~\ref{thm:2N}]
	Rewrite
	\begin{equation*} 
	\widetilde \beta_N = \frac{1- \widehat G_{N} (1-\delta)}{\int_{\delta^r}^{\delta} (1- \widehat G_{N} (1- x)) \frac{\d x}{x}}.
	\end{equation*}
	Split $\widetilde \beta_N - \beta =  \widetilde D^{-1} (\sum_{i=1}^{4} I_i + \sum_{i=1}^{4} R_i)$, where $I_i$, $i=1, \dots, 4$, are defined in \eqref{def:Ii} and
	\begin{eqnarray*}
		R_1 &:=& \beta \int_{\delta^r}^\delta ( \widehat G_N(1-x) - G_N (1-x)) \frac{\d x}{x},\qquad
		R_2 := G_N(1-\delta) - \widehat G_N (1-\delta),\\
		R_3 &:=& \beta \int_0^{\delta^r} (G(1-x) - G_N (1-x)) \frac{\d x}{x}, \qquad R_4 := \beta \int_0^{\delta^r} (1- G(1-x)) \frac{\d x}{x}
	\end{eqnarray*}
	and
	\begin{eqnarray*}
		\widetilde D := \int_{\delta^r}^\delta (1 - \widehat G_N(1-x)) \frac{\d x}{x} = D - \frac{1}{\beta} (R_1+R_3+R_4)
	\end{eqnarray*}
	with $D$ given by \eqref{def:D}.
	By Proposition \ref{prop:2}, $(N \delta^{-\beta})^{1/2} R_2 = o_p(1)$ and
	$(N \delta^{-\beta})^{1/2} R_1 = o_p(1)$.
	In view of \eqref{varI10}, we have $\E | (N \delta^{-\beta})^{1/2} R_3|^2 \le C \delta^{r \beta} = o (1)$ and so
	$(N \delta^{-\beta})^{1/2} R_3 = o_p (1)$.
	Finally, $(N\delta^{-\beta})^{1/2} R_4 = o(1)$ as $N \delta^{(2r-1) \beta} \to 0$.
\end{proof}

\begin{proof} [Proof of Corollary~\ref{thm:2}]
Let
$\chi := T^{-\min\{p-1,p/2\}}$, $\chi' := \chi^{1/\min\{p-1,p/2\}} = T^{-1}$.
Then \eqref{cond1}--\eqref{cond2} agree with \eqref{condN} and the result follows from Theorem \ref{thm:2N}.
\end{proof}

\begin{proof}[Proof of Corollary~\ref{cor:1}]
	Let $K_N = \sum_{i=1}^{N} \1 (a_i > 1 - \delta)$. Since $\operatorname{Var} (K_N) \le N (1-G(1-\delta))$ and $N (1-G(1-\delta)) \to \infty$, Markov's inequality yields
	$$
	\frac{K_N}{N (1 - G(1-\delta))} \to_p 1,
	$$
	consequently, $(N\delta^\beta)^{-1} K_N \to_p \kappa$.
	By Proposition~\ref{prop:2}, we have $(N \delta^{\beta})^{-1} (\widetilde K_N - K_N) = o_p (1)$. We conclude that $(N \delta^\beta)^{-1} \widetilde K_N \to_p \kappa$.
\end{proof}

\section*{Acknowledgments}

The authors are grateful to two anonymous referees for criticisms and helpful suggestions.
We thank Marijus Vai\v ciulis for helping us with the choice of the threshold in the simulation experiment.
Vytaut{\.e} Pilipauskait{\.e} acknowledges the financial support from the project ``Ambit fields: probabilistic properties and statistical inference'' funded by Villum Fonden.

\section*{Data availability statement}

Data sharing is not applicable to this article as no new data were created or analysed in this study.

\end{document}